\documentclass[reqno]{amsart}

\usepackage{amsmath}
\usepackage{amssymb}
\usepackage{amsthm}
\usepackage{hyperref}
\usepackage{mathrsfs}

\usepackage[backend=bibtex,style=alphabetic,firstinits,url=false,doi=true]{biblatex}
\numberwithin{equation}{section}
\newtheorem{theorem}[equation]{Theorem}

\newtheorem{proposition}[equation]{Proposition}
\newtheorem{corollary}[equation]{Corollary}

\theoremstyle{definition}
\newtheorem{example}[equation]{Example}
\newtheorem{definition}[equation]{Definition}

\newcommand{\define}[1]{\textit{#1}}
\newcommand{\nbar}{\vert\!\vert}

\newcommand{\ultra}[1]{\mathsf{{#1}}}
\renewcommand{\epsilon}{\varepsilon}
\renewcommand{\emptyset}{\varnothing}

\newcommand{\ipr}[1]{\mathrm{IP}_{{#1}}}

\DeclareMathOperator*{\dlim}{UD-lim}

\DeclareMathOperator{\conv}{\ast}

\DeclareMathOperator{\fs}{FS}
\DeclareMathOperator{\fu}{FU}

\DeclareMathOperator{\lp}{L}
\DeclareMathOperator{\hj}{HJ}

\DeclareMathOperator{\ip}{IP}
\DeclareMathOperator{\mc}{MC}
\DeclareMathOperator{\amc}{AMC}
\DeclareMathOperator{\aip}{AIP}

\newcommand{\upperdens}{\mathrm{d}^*}

\DeclareMathOperator{\FS}{FS}

\title{Polynomial recurrence with large intersection over countable fields}
\author{Vitaly Bergelson}
\email{vitaly@math.ohio-state.edu}
\thanks{The first author gratefully acknowledges the support of the NSF under grant DMS-1162073.}
\author{Donald Robertson}
\email{robertson@math.ohio-state.edu}
\address{Department of Mathematics\\
  The Ohio State University\\
  231 West 18th Avenue\\
  Columbus\\
  OH 43210-1174\\
  USA}

\date{\today{}}

\begin{document}

\begin{abstract}
We give a short proof of polynomial recurrence with large intersection for additive actions of finite-dimensional vector spaces over countable fields on probability spaces, improving upon the known size and structure of the set of strong recurrence times.
\end{abstract}

\maketitle

\section{Introduction}

Let $F$ be a countable field and let $\phi \in F[x]$ have zero constant term.
Given a measure preserving action $T$ of the additive group of $F$ on a probability space $(X,\mathscr{B},\mu)$, a set $B \in \mathscr{B}$ and $\epsilon > 0$, we will show that, for any $\epsilon > 0$ the set
\begin{equation*}
\{ u \in F : \mu(B \cap T^{\phi(u)}B) \ge \mu(B)^2 - \epsilon \}
\end{equation*}
of strong recurrence times is large, in the sense of being $\ip^*_r$ up to a set of zero Banach density.
(These notions of size are defined below.)
In fact, we prove a more general result regarding strong recurrence for commuting actions of countable fields along polynomial powers.
This strengthens and extends recent results from \cite{MWcountableFields} regarding actions of fields having finite characteristic.
Here are the relevant definitions.

\begin{definition}
Let $G$ be an abelian group.
An \define{IP set} or \define{finite sums set} in $G$ is any subset of $G$ containing a set of the form
\begin{equation*}
\fs(x_1,x_2,\dots) := \bigg\{ \sum_{n \in \alpha} x_n : \emptyset \ne \alpha \subset \mathbb{N}, |\alpha| < \infty \bigg\}
\end{equation*}
for some sequence $n \mapsto x_n$ in $G$.
Given $r \in \mathbb{N}$, an \define{IP$_r$ set} in $G$ is any subset of $G$ containing a set of the form
\begin{equation*}
\fs(x_1,x_2,\dots,x_r) := \bigg\{ \sum_{n \in \alpha} x_n : \emptyset \ne \alpha \subset \{ 1,\dots,r \} \bigg\}
\end{equation*}
for some $x_1,\dots,x_r$ in $G$.
A subset of $G$ is \define{$\ip^*$} if its intersection with every $\ip$ set in $G$ is non-empty, and \define{$\ip^*_r$} if its intersection with every $\ip_r$ set is non-empty.
The term $\ip$ was introduced in \cite{MR531271}, the initials standing for ``idempotence'' or ``infinite-dimensional parallelopiped'' and $\ip^*_r$ sets were introduced in \cite{MR833409}.
The \define{upper Banach density} of a subset $S$ of $G$ is defined by
\begin{equation*}
\upperdens(S) = \sup \left\{ \upperdens_\Phi(S) : \Phi \textup{ a F\o{}lner sequence in } G \right\}
\end{equation*}
where
\begin{equation*}
\upperdens_\Phi(S) = \limsup_{N \to \infty} \frac{|S \cap \Phi_N|}{|\Phi_N|}
\end{equation*}
and a \define{F\o{}lner sequence} is a sequence $N \mapsto \Phi_N$ of finite, non-empty subsets of $G$ such that
\begin{equation*}
\lim_{N \to \infty} \frac{|(g + \Phi_N) \cap \Phi_N|}{|\Phi_N|} = 1
\end{equation*}
for all $g$ in $G$.
Lastly, $S \subset G$ is said to be \define{almost $\ip^*$} (written $\aip^*)$ if it is of the form $A \setminus B$ where $A$ is $\ip^*$ and $\upperdens(B) = 0$, and said to be \define{almost $\ip^*_r$} (written $\aip^*_r$) if it is of the form $A \backslash B$ where $A$ is $\ip^*_r$ and $\upperdens(B) = 0$.
\end{definition}

Although when $G = \mathbb{Z}$ any $\ip$ set with non-zero generators is unbounded, this is not the case in general.
For example, if $G = \mathbb{Q}$ then the $\ip$ set generated by the sequence $n \mapsto 1/n^2$ remains bounded.

To state our result we recall some definitions from \cite{MR2145566}.
Fix a countable field $F$.
By a \define{monomial} we mean a mapping $F^n \to F$ of the form $(x_1,\dots,x_n) \mapsto a x_1^{d_1} \cdots x_n^{d_n}$ for some $a \in F$ and integers $d_1,\dots,d_n \ge 0$ not all zero.
Let $V$ and $W$ be finite-dimensional vector spaces over $F$.
A mapping $F^n \to W$ is a \define{polynomial} if it is a linear combination of vectors with monomial coefficients.
A mapping $V \to W$ is a \define{polynomial} if, in terms of a basis of $V$ over $F$, it is a polynomial mapping $F^n \to W$.
Here is our main result.

\begin{theorem}
\label{thm:fieldsPolyRec}
Let $W$ be a finite-dimensional vector space over a countable field $F$ and let $T$ be an action of the additive group of $W$ on a probability space $(X,\mathscr{B},\mu)$.
For any polynomial $\phi : F^n \to W$, any $B \in \mathscr{B}$ and any $\epsilon > 0$ the set
\begin{equation}
\label{eqn:fieldLargeRec}
\{ u \in F^n : \mu(B \cap T^{\phi(u)} B) > \mu(B)^2 - \epsilon \}
\end{equation}
is $\aip^*_r$ for some $r \in \mathbb{N}$.
\end{theorem}

Our result implies in particular that \eqref{eqn:fieldLargeRec} is syndetic.
In fact, as we will show in Section~\ref{sec:proof}, we have generalized \cite[Corollary~5]{MWcountableFields}, where, in the finie characteristic case, the set \eqref{eqn:fieldLargeRec} is shown to belong to every essential idempotent ultrafilter on $F$.
This latter notion of largeness, introduced in \cite{MR2353901}, lies between syndeticity and $\aip^*_r$.

The conclusion of Theorem~\ref{thm:fieldsPolyRec} is of an additive nature: the notion of being $\aip^*_r$ is only related to the additive structure of $F^n$.
It is natural to ask, when $n = 1$, whether \eqref{eqn:fieldLargeRec} is also large in terms of the multiplicative structure of $F$.
We address this question in Section~\ref{sec:multiplicative}, proving that in fact \eqref{eqn:fieldLargeRec} intersects any multiplicatively central set that has positive upper Banach density.
Multiplicatively central sets are defined in Section~\ref{sec:multiplicative} and upper Banach density is as defined above.

This result is proved in Section~\ref{sec:proof}.
In Section~\ref{sec:iprSets} we prove the facts we will need about $\ip^*_r$ sets.
Finally, in Section~\ref{sec:multiplicative} we relate the largeness of the set \eqref{eqn:fieldLargeRec} to the multiplicative structure of $F$ in the case $n = 1$.

We would like to thank R. McCutcheon for communicating to us his result used at the end of Section~\ref{sec:proof}.

\section{Finite IP sets}
\label{sec:iprSets}

Let $\mathscr{F}$ be the collection of all finite, non-empty subsets of $\mathbb{N}$.
Write $\alpha < \beta$ for elements of $\mathscr{F}$ if $\max \alpha < \min \beta$.
A subset of $\mathscr{F}$ is an $\fu$ set if it contains a sequence $\alpha_1 < \alpha_2 < \cdots$ from $\mathscr{F}$ and all finite unions of sets from the sequence.
Write $\mathscr{F}_r$ for all finite, non-empty subsets of $\{ 1,\dots, r\}$.
A subset of $\mathscr{F}_r$ (or of $\mathscr{F}$) is an $\fu_s$ set if it contains sets $\alpha_1 < \cdots < \alpha_s$ from $\mathscr{F}_r$ (or from $\mathscr{F}$) and all finite unions.
For any $\ip_r$ set $A \supset \fs(x_1,\dots,x_r)$ in an abelian group $G$ there is a map $\mathscr{F}_r \to G$ given by $\alpha \mapsto \sum \{x_i : i \in \alpha \}$, and for any $\ip$ set in $G$ there is a map $\mathscr{F} \to G$ defined similarly.

Furstenberg and Katznelson \cite{MR833409} showed that any $\ip_r^*$ set $A$ in $\mathbb{Z}$ satisfies
\begin{equation*}
\liminf_{N \to \infty} \frac{|A \cap \{1 ,\dots, N\}|}{N} \ge \frac{1}{2^{r-1}}
\end{equation*}
so for any $r \in \mathbb{N}$ one can construct an $\ip^*$ set that is not $\ip^*_r$.
The set $k\mathbb{N}$, with $k$ large enough, is one such example.
As the following example shows, by removing well-spread $\ip_r$ sets from $\mathbb{Z}$, it is possible to construct a set that is $\ip^*$ but never $\ip^*_r$.

\begin{example}
Let $A_r$ be the $\ip_r$ set with generators $x_1 = \cdots = x_r = 2^{2^r}$ so that $A_r = \{ i \cdot 2^{2^r} : 1 \le i \le r \}$.
Let $A$ be the union of all the $A_r$.
We claim that $A$ cannot contain an $\ip$ set, from which it follows that $\mathbb{N}\backslash A$ is $\ip^*$.
Since $A$ contains $\ip_r$ sets for arbitrarily large $r$ we also have that $\mathbb{N} \backslash A$ is not $\ip^*_r$ for any $r$.

Suppose that $x_n$ is a sequence generating an $\ip$ set in $A$.
If one can find $x_i \in A_r$ and $x_j \in A_s$ with $r < s$ then $x_j + x_i$ does not belong to $A$ because the gaps in $A_s$ are larger than the largest element in $A_r$.
On the other hand, if all $x_i$ belong to the same $A_r$ then some combination of them is not in $A$ because the gap between $A_r$ and $A_{r+1}$ is too large.
\end{example}

A family $\mathscr{S}$ of subsets of $G$ is said to have the \define{Ramsey property} if $S_1 \cup S_2$ belonging to $\mathscr{S}$ always implies that at least one of $S_1$ or $S_2$ contains a member of $\mathscr{S}$.
It follows from the reformulation of 
Hindman's theorem \cite{MR0349574}, stated below, that the collection of all $\ip$ subsets of a group $G$ has the Ramsey property.
A \define{coloring} of a set $A$ is any map $c : A \to \{ 1,\dots, k \}$ for some $k \in \mathbb{N}$.
Given a coloring of $A$, a subset $B$ is then called \define{monochromatic} if $c$ is constant on $B$.

\begin{theorem}[{\cite[Corollary~3.3]{MR0349574}}]
For any coloring of $\mathscr{F}$ one can find $\alpha_1 < \alpha_2 < \cdots$ in $\mathscr{F}$ such that the collection of all finite unions of the sets $\alpha_i$ is monochromatic.
\end{theorem}

Given a family $\mathscr{I}$ of subsets of $G$, the \define{dual family} of $\mathscr{S}$ is the collection $\mathscr{S}^*$ of subsets of $G$ that intersect every member of $\mathscr{S}$ non-emptily.
Taking $\mathscr{S}$ to consist of all $\ip$ sets, one can deduce that the intersection of an $\ip^*$ set with an $\ip$ set contains an $\ip$ set and that the intersection of two $\ip^*$ sets is again $\ip^*$.
The collection of all $\ip_r$ sets does not have the Ramsey property, but there is a suitable replacement that allows one to deduce results about $\ip^*_r$ sets similar to the ones for $\ip^*$ sets mentioned above.

\begin{proposition}
For any $s$ and $k$ in $\mathbb{N}$ there is an $r$ such that any $k$-coloring of any $\ip_r$ set yields a monochromatic $\ip_s$ set.
\begin{proof}
Suppose to the contrary that one can find $s$ and $k$ in $\mathbb{N}$ such that, for any $r$ there is a $k$-coloring of an $\ip_r$ set $A_r$ having no monochromatic $\ip_s$ subset.
This coloring of $A_r$ gives rise to a coloring $c_r$ of $\mathscr{F}_r$ via the canonical map $\mathscr{F}_r \to A_r$. That no $A_r$ contains a monochromatic $\ip_s$ set implies that no $\mathscr{F}_r$ contains a monochromatic $\fu_s$ set. We now use Hindman's theorem to reach a contradiction.

Let $\alpha_i$ be an enumeration of $\mathscr{F}$. We construct a coloring $c : \mathscr{F} \to \{ 1,\dots, k \}$ by induction on $i$. To begin note that $\alpha_1 \in \mathscr{F}_r$ whenever $r > \max \alpha_1$ so we can find a strictly increasing sequence $r(1,n)$ in $\mathbb{N}$ such that $c_{r(1,n)}(\alpha_1)$ takes the same value for all $n$. Put $c(\alpha_1) = c_{r(1,n)}(\alpha_1)$. Now, assuming that we have found a strictly increasing sequence $r(i,n)$ such that, for each $1 \le j \le i$ the color $c_{r(i,n)}(\alpha_j)$ is constant in $n$ and equal to $c(\alpha_j)$, choose a strictly increasing subsequence $r(i+1,n)$ of $r(i,n)$ such that $c_{r(i+1,n)}(\alpha_{i+1})$ is constant and let this value be $c(\alpha_{i+1})$. The colors of $\alpha_1,\dots,\alpha_i$ are unchanged and the induction argument is concluded.

By Hindman's theorem we can find $\beta_1 < \cdots < \beta_s$ in $\mathscr{F}$ such that $B = \fu(\beta_1,\dots,\beta_s)$ is monochromatic, meaning $c$ is constant on $B$.
Choose $i$ such that $B \subset \{ \alpha_1,\dots, \alpha_i \}$ and then choose $n$ so large that $r(i,n) > \max \beta_s$.
It follows that $B \subset \mathscr{F}_{r(i,n)}$ is monochromatic because $c_{r(i,n)}(\beta) = c(\beta)$ for all $\beta \in B$.
Thus $\mathscr{F}_{r(i,n)}$ contains a monochromatic $\fu_s$ set, which is a contradiction.
\end{proof}
\end{proposition}

With this version of partition regularity for $\ip_r$ sets we can deduce some facts about $\ip_r^*$ sets.

\begin{proposition}
Given any $s \in \mathbb{N}$ there is some $r \in \mathbb{N}$ such that any $\ip_s^*$ set intersects any $\ip_r$ set in an $\ip_s$ set.
\begin{proof}
Let $A$ be an $\ip_s^*$ set and choose by the previous proposition some $r$ such that any two-coloring of an $\ip_r$ set yields a monochromatic $\ip_s$ set. Let $B$ be an $\ip_r$ set.
One of $B \cap A$ and $B \backslash A$ contains an $\ip_s$ set.
It cannot be $B \backslash A$ because $A$ is $\ip_s^*$ and disjoint from it.
Thus $A \cap B$ contains an $\ip_s$ set as desired.
\end{proof}
\end{proposition}

\begin{proposition}
\label{prop:ipstarFilter}
Given any $r,s$ in $\mathbb{N}$ there is some $\alpha(r,s) \in \mathbb{N}$ such that if $A$ is $\ip_r^*$ and $B$ is $\ip^*_s$ then $A \cap B$ is $\ip^*_{\alpha(r,s)}$.
\begin{proof}
Let $A$ be $\ip^*_r$ and let $B$ be $\ip^*_s$ with $r \ge s$.
Choose $q$ so large that $A \cap C$ contains an $\ip_r$ set whenever $C$ is an $\ip_q$ set.
This is possible by the previous result.
Since $A \cap C$ contains an $\ip_r$ set and $r \ge s$ the set $(A \cap C) \cap B$ must be non-empty.
Since $C$ was arbitrary $A \cap B$ is an $\ip_q^*$ set.
Put $\alpha(r,s) = q$.
\end{proof}
\end{proposition}

\section{Proof of Theorem~\ref{thm:fieldsPolyRec}}
\label{sec:proof}

First we note that we may assume, by restricting our attention to the sub-$\sigma$-algebra generated by the orbit of $B$, that the probability space $(X,\mathscr{B},\mu)$ is separable.

We begin with a corollary of the Hales-Jewett theorem.
For any $n \in \mathbb{N}$ write $[n] = \{ 1,\dots,n\}$.
Write $\mathcal{P}A$ for the set of all subsets of a set $A$.
Recall that, given $k,m \in \mathbb{N}$, a \define{combinatorial line} in $[k]^{[m]}$ is specified by a partition $U_0 \cup U_1$ of $\{1,\dots,m\}$ with $U_1 \ne \emptyset$ and a function $\varphi : U_0 \to [k]$, and consists of all functions $[m] \to [k]$ that extend $\varphi$ and are constant on $U_1$.
With these definitions we can state the Hales-Jewett theorem.

\begin{theorem}[\cite{MR0143712}]
For every $d,t \in \mathbb{N}$ there is $r = \hj(d,t) \in \mathbb{N}$ such that for any $t$-coloring of $[d]^{[r]}$ one can find a monochromatic combinatorial line.
\end{theorem}

\begin{corollary}
\label{cor:hjSets}
For any $d,t \in \mathbb{N}$ there is $r \in \mathbb{N}$ such that any $t$-coloring
\begin{equation*}
(\mathcal{P}\{1,\dots,r\})^d \to \{1,\dots,t\}
\end{equation*}
contains a monochromatic configuration of the form
\begin{equation}
\label{eqn:hjSets}
\{ (\alpha_1 \cup \eta_1,\dots,\alpha_d \cup \eta _d) : (\eta_1,\dots,\eta_d) \in \{ \emptyset, \gamma \}^d \}
\end{equation}
for some $\gamma,\alpha_1,\dots,\alpha_d \subset \{1,\dots,r\}$ with $\gamma$ non-empty and $\gamma \cap \alpha_i = \emptyset$ for each $1 \le i \le d$.
\end{corollary}
\begin{proof}
Let $r = \hj(2^d,t)$.
Define a map $\psi : [2^d]^{[r]} \to (\mathcal{P}[r])^d$ by declaring $\psi(w) = (\alpha_1,\dots,\alpha_d)$ where $\alpha_i$ consists of those $j \in [r]$ for which the binary expansion of $w(j)-1$ has a 1 in the $i$th position.
Combinatorial lines in $[2^d]^{[r]}$ correspond via this map to configurations of the form \eqref{eqn:hjSets} in $(\mathcal{P}[r])^d$.
\end{proof}

We use the above version of the Hales-Jewett theorem to derive the following topological recurrence result.
Given $n \in \mathbb{N}$ and a ring $R$, by a \define{monomial mapping} from $R^n$ to $R$ we mean any map of the form $(x_1,\dots,x_n) \mapsto ax_1^{d_1} \cdots x_n^{d_n}$ for some $a \in R$ and some $d_1,\dots,d_n \ge 0$ not all zero.

\begin{proposition}[cf {\cite[Theorem~7.7]{MR2757532}}]
\label{lem:iprOnMetric}
Let $R$ be a commutative ring and let $T$ be an action of the additive group of $R$ on a compact metric space $(X,\mathsf{d})$ by isometries.
For any monomial mapping $\phi : R^n \to R$, any $x \in X$ and any $\epsilon > 0$ there is $r \in \mathbb{N}$ such that the set
\begin{equation*}
\{ u \in R^n : \mathsf{d}(T^{\phi(u)} x, x) < \epsilon \}
\end{equation*}
is $\ipr{r}^*$.
\end{proposition}
\begin{proof}
Write $\phi(x_1,\dots,x_n) = a x_1^{d_1} \cdots x_n^{d_n}$ for some $a \in R$ and some $d_i \ge 0$ not all zero.
Let $d = d_1 + \cdots + d_n$.
Put $e_0 = 0$ and $e_i = d_1 + \cdots + d_i$ for each $1 \le i \le n$.
Fix $x \in X$ and $\epsilon > 0$.
Let $V_1,\dots,V_t$ be a cover of $X$ by balls of radius $\epsilon/2^d$.
Let $r = r(d,t)$ be as in Corollary~\ref{cor:hjSets}.
Fix $u_1,\dots,u_r$ in $R^n$.
Given $\alpha \subset \{ 1,\dots,r \}$ write $u_\alpha$ for $\Sigma \{ u_i : i \in \alpha \}$ and $u_\alpha(i)$ for the $i$th coordinate of $u_\alpha$.
By choosing for each $(\alpha_1,\dots,\alpha_d) \in (\mathcal{P}\{ 1,\dots,r\})^d$ the minimal $1 \le i \le t$ such that
\begin{equation*}
T(a u_{\alpha_1}(1) \cdots u_{\alpha_{e_1}}(1) \cdots u_{\alpha_{e_{n-1}+1}}(n) \cdots u_{\alpha_{e_n}}(n)) x \in V_i
\end{equation*}
we obtain via Theorem~\ref{cor:hjSets} sets $\alpha_1,\dots,\alpha_d,\gamma \subset \{ 1,\dots,r \}$ with $\gamma$ non-empty and disjoint from all $\alpha_i$ which, combined with the expansion
\begin{equation*}
a u_\gamma(1)^{d_1} \cdots u_\gamma(n)^{d_n} = a \prod_{k=1}^{n} \prod_{i=e_{k-1}+1}^{e_{k}} u_\gamma(k) + u_{\alpha_i}(k) - u_{\alpha_i}(k)
\end{equation*}
and the fact that $T$ is an isometry, yields $\mathsf{d}(T^{\phi(u_\gamma)}x,x) < \epsilon$ as desired.
\end{proof}

Let $G$ be an abelian group.
Actions $T_1$ and $T_2$ of $G$ are said to \define{commute} if $T_1^g T_2^h = T_2^h T_1^g$ for all $g,h \in G$.
As we now show, iterating the previous result yields a version for commuting actions of rings.

\begin{corollary}
\label{cor:iprPolys}
Let $R$ be a commutative ring and let $T_1,\dots,T_k$ be commuting actions of the additive group of $R$ on a compact metric space $(X,\mathsf{d})$ by isometries.
For any monomial mappings $\phi_1,\dots,\phi_k : R^n \to R$, any $x \in X$ and any $\epsilon > 0$, there is $r \in \mathbb{N}$ such that
\begin{equation}
\label{eqn:metricPolyReturns}
\{ u \in R^n : \mathsf{d}(T_1^{\phi_1(u)} \cdots T_k^{\phi_k(u)} x, x) < \epsilon \}
\end{equation}
is $\ip^*_r$.
\end{corollary}
\begin{proof}
Fix $1 \le i \le k$.
By applying Proposition~\ref{lem:iprOnMetric} to the $R$ action $r \mapsto T_i^r$, we can find $r_i \in \mathbb{N}$ such that
\begin{equation*}
Z_i = \{ u \in R^n : \mathsf{d}(T_i^{\phi_i(u)}x,x) < \epsilon/k \}
\end{equation*}
is $\ip^*_{r_i}$.
By Proposition~\ref{prop:ipstarFilter}, the intersection $Z_1 \cap \cdots \cap Z_k$ is $\ip^*_r$ for some $r \in \mathbb{N}$.
Since the $T_i$ are isometries, it follows that \eqref{eqn:metricPolyReturns} is $\ip^*_r$ as desired.
\end{proof}

Combining the preceding lemma with the following facts from \cite{MR2145566} will lead to a proof of Theorem~\ref{thm:fieldsPolyRec}.
Let $\phi : V \to W$ be a polynomial and let $T$ be an action of $W$ on a probability space $(X,\mathscr{B},\mu)$.
Assume that $\phi V$ spans $W$.
As in \cite{MR2145566}, say that $f$ in $\lp^2(X,\mathscr{B},\mu)$ is \define{weakly mixing} for $(T,\phi)$ if $\dlim_v \langle T^{\phi (v)} f, g \rangle = 0$ for all $g$ in $\lp^2(X,\mathscr{B},\mu)$, where $\dlim$ denotes convergence with respect to the filter of sets whose complements have zero upper Banach density.
This is the same as strong Ces\`{a}ro convergence along every F\o{}lner sequence in $V$.
Call $f \in \lp^2(X,\mathscr{B},\mu)$ \define{compact} for $T$ if $\{ T^{v} f : v \in V \}$ is pre-compact in the norm topology.
Denote by $\mathscr{H}_\mathrm{wm}(T,\phi)$ the closed subspace of $\lp^2(X,\mathscr{B},\mu)$ spanned by functions that are weakly mixing for $(T,\phi)$, and let $\mathscr{H}_\mathrm{c}(T)$ be the closed subspace of $\lp^2(X,\mathscr{B},\mu)$ spanned by functions compact for $T$.
We have $\lp^2(X,\mathscr{B},\mu) = \mathscr{H}_\mathrm{c}(T) \oplus \mathscr{H}_\mathrm{wm}(T,\phi)$ by \cite[Theorem~3.17]{MR2145566}.

\begin{proof}[Proof of Theorem~\ref{thm:fieldsPolyRec}]
Write $\phi = \phi_1 w_1 + \cdots + \phi_k w_k$ where the $\phi_i$ are monomials $F^n \to F$ and the $w_i$ belong to $V$.
Fix $B$ in $\mathscr{B}$ and $\epsilon > 0$.
Let $f = P1_B$ be the orthogonal projection of $1_B$ on $\mathscr{H}_\mathrm{c}(T)$.
Let $\Omega$ be the orbit closure of $f$ in the norm topology under $T$.
Since $f$ is compact, $\Omega$ is a compact metric space.
Applying Lemma~\ref{cor:iprPolys} to the $F$ actions $x \mapsto T^{xw_i}$ and monomials $\phi_i$ for $1 \le i \le k$, we see that
\begin{equation*}
\{ u \in F^n : \nbar f - T^{\phi(u)} f \nbar < \epsilon/2 \}
\end{equation*}
is $\ip^*_r$.
We have
\begin{equation*}
\langle T^{\phi(u)} 1_B, 1_B \rangle = \langle T^{\phi(u)} f, 1_B \rangle + \langle T^{\phi(u)}(1_B - f), 1_B \rangle
\end{equation*}
so the set
\begin{equation*}
\{ u \in F^n : \langle T^{\phi(u)} 1_B, 1_B \rangle \ge \langle f, 1_B \rangle - \epsilon/2 + \langle T^{\phi(u)}(1_B - f), 1_B \rangle
\end{equation*}
is $\ip^*_r$.
Since $1_B - f$ is weakly mixing for $(T,\phi)$ the set
\begin{equation*}
\{ u \in F^n : \langle T^{\phi(u)} 1_B, 1_B \rangle \ge \langle f, 1_B \rangle - \epsilon \}
\end{equation*}
is $\aip^*_r$.
Thus \eqref{eqn:fieldLargeRec} is $\aip^*_r$ by
\begin{equation*}
\langle f, 1_B \rangle = \langle P1_B, P1_B \rangle \langle 1,1\rangle \ge \langle P1_B, 1 \rangle^2 = \mu(B)^2
\end{equation*}
as desired.
\end{proof}

We obtain as a corollary the following result from \cite{MWcountableFields}.
An ultrafilter $\ultra{p}$ on an abelian group $G$ is \define{essential} if it is idempotent and $\upperdens(A) > 0$ for all $A \in \ultra{p}$.

\begin{corollary}[{\cite[Corollary~5]{MWcountableFields}}]
\label{cor:mw}
Let $F$ be a countable field of finite characteristic and let $p : F \to F^n$ be a polynomial mapping.
For any action $T$ of $F^n$ on a probability space $(X,\mathscr{B},\mu)$, any $B$ in $\mathscr{B}$ and any $\epsilon > 0$ the set
\begin{equation}
\label{eqn:mwPolyRec}
\{ x \in F : \mu(B \cap T^{p(x)}B) \ge \mu(B)^2 - \epsilon \}
\end{equation}
belongs to every essential idempotent ultrafilter.
\end{corollary}
\begin{proof}
It follows from the proof of Theorem~\ref{thm:fieldsPolyRec} that \eqref{eqn:mwPolyRec} is of the form $A \setminus B$ where $A$ is $\ip^*_r$ for some $r \in \mathbb{N}$ and $B$ has zero upper Banach density.
Any $\ip^*_r$ subset of $G$ is $\ip^*$ and therefore belongs to every idempotent ultrafilter on $G$, so $A$ certainly belongs to every essential ultrafilter on $G$.
By the filter property, removing from $A$ a set of zero upper Banach density does not change this fact, because every set in an essential idempotent has positive upper Banach density.
\end{proof}

It was recently shown by R. McCutcheon that there are sets belonging to every essential, idempotent ultrafilter that are not $\aip^*$.
Thus our result constitues a genuine strengthening of Corollary~\ref{cor:mw}.

\section{Multiplicative structure}
\label{sec:multiplicative}

According to Theorem~\ref{thm:fieldsPolyRec} the set \eqref{eqn:fieldLargeRec} is large in terms of the additive structure of $F^n$.
In this section we connect the largeness of \eqref{eqn:fieldLargeRec} when $n = 1$ to the multiplicative structure of $F$ by showing that \eqref{eqn:fieldLargeRec} is almost an $\mc^*$ subset of $F$.
Here $\mc$ stands for \define{multiplicatively central} and a set is $\mc^*$ if its intersection with every multiplicatively central set is non-empty.

To define what a multiplicatively central set is, recall that, given a commutative ring $R$, we can extend the multiplication on $R$ to a binary operation $\conv$ on the set $\beta R$ of all ultrafilters on $R$ by
\begin{equation*}
\ultra{p} \conv \ultra{q} = \{ A \subset R : \{ u \in R : Au^{-1} \in \ultra{p} \} \in \ultra{q} \}
\end{equation*}
for all $\ultra{p},\ultra{q} \in \beta R$.
One can check that this makes $\beta R$ a semigroup.
It is also possible to equip $\beta R$ with a compact, Hausdorff topology with respect to which the binary operation is right continuous.
See \cite{MR2052273} or \cite{MR2893605} for the details of these constructions.
A subset $A$ of $R$ is then called \define{multiplicatively central} or \define{MC} if it belongs to an ultrafilter that is both idempotent and contained in a minimal right ideal of $\beta R$.
The following version of \cite[Theorem~3.5]{MR1305896} relates $\ip_r$ sets in $R$ to multiplicatively central sets.

\begin{proposition}
Let $R$ be a commutative ring and let $A \subset R$ be a multiplicatively central set.
For every $r \in \mathbb{N}$ one can find $x_1,\dots,x_r$ in $R$ such that $\FS(x_1,\dots,x_r) \subset A$.
\end{proposition}
\begin{proof}
Consider the family $T$ of ultrafilters $\ultra{p}$ on $R$ having the property that every set in $\ultra{p}$ contains an $\ip_r$ set for every $r \in \mathbb{N}$.
We claim that $T$ is a two-sided ideal in $\beta R$.
Indeed fix $\ultra{p} \in T$ and $\ultra{q} \in \beta R$.
We need to prove that $\ultra{p} \conv \ultra{q}$ and $\ultra{q} \conv \ultra{p}$ belong to $T$.

For the former, fix $B \in \ultra{p} \conv \ultra{q}$ and $r \in \mathbb{N}$.
We can find $u \in R$ such that $Bu^{-1} \in \ultra{p}$ so $Bu^{-1}$ contains $\FS(x_1,\dots,x_r)$ for some $x_1,\dots,x_r$ in $R$.
This immediately implies that $\FS(x_1u,\dots,x_ru) \subset B$ as desired.

For the latter, fix $B \in \ultra{q} \conv \ultra{p}$ and $r \in \mathbb{N}$.
We can find $x_1,\dots,x_r$ in $R$ such that $\FS(x_1,\dots,x_r) \subset \{ u \in R : Bu^{-1} \in \ultra{q} \}$.
But by the filter property
\begin{equation}
\label{eqn:finiteIntersection}
\cap \{ Bu^{-1} : u \in \FS(x_1,\dots,x_r) \} \in \ultra{q}
\end{equation}
and choosing $a$ from this intersection gives $\FS(ax_1,\dots,ax_r) \subset B$.

Our set $A$ is multiplicatively central so it is contained in some idempotent ultrafilter $\ultra{p}$ that belongs to a minimal right ideal $S$.
Since $T$ is also a right ideal $S \subset T$ and $\ultra{p} \in T$ as desired.
\end{proof}

Note that it is not possible to prove this way that multiplicatively central sets contain $\ip$ sets, as that would require an infinite intersection in \eqref{eqn:finiteIntersection}.
In fact, as shown in \cite[Theorem~3.6]{MR1305896}, there are multiplicatively central sets in $\mathbb{N}$ that do not contain $\ip$ sets.

We say that a subset of $R$ if $\mc^*$ if its intersection with every multiplicatively central set is non-empty.
As noted in \cite{MR2757532}, the preceding result implies that every $\ip^*_r$ set is $\mc^*$.
Call a set $\amc^*$ (with A again standing for ``almost'') if it is of the form $A \setminus B$ where $A$ is $\mc^*$ and $B$ has zero upper Banach density in $(R,+)$.
The following result is then an immediate consequence of Theorem~\ref{thm:fieldsPolyRec}.

\begin{theorem}
Let $F$ be a countable field and let $T$ be an action of the additive group of $F$ on a probability space $(X,\mathscr{B},\mu)$.
For any polynomial $\phi \in F[x]$, any $B \in \mathscr{B}$ and any $\epsilon > 0$ the set
\begin{equation}
\{ u \in F : \mu(B \cap T^{\phi(u)} B) > \mu(B)^2 - \epsilon \}
\end{equation}
is $\amc^*$.
\end{theorem}

We conclude by mentioning that all $\amc^*$ sets have positive upper Banach density in $(F,+)$.
This follows from the fact that every $\mc^*$ set belongs to every minimal multiplicative idempotent, and a straightforward generalization of \cite[Theorem~5.6]{MR982232}, which guarantees the existence of a minimal idempotent for $\conv$ all of whose members have positive upper Banach density in $(F,+)$.

\printbibliography

\end{document}